\documentclass[12pt]{article}   
\usepackage{graphicx,psfrag}
\usepackage[table]{xcolor}
\usepackage{color}
\usepackage{tikz}
\usepackage{array}
\usepackage{colortbl}
\usepackage{amsmath,amsfonts}
\usepackage{amssymb,amsthm}

\usepackage{epstopdf}

\usepackage{tabularx}
\newtheorem{theorem}{Theorem}[section]
\newtheorem{corollary}[theorem]{Corollary}
\newtheorem{proposition}[theorem]{Proposition}

\newtheorem{example}[theorem]{Example}
\newtheorem{lemma}[theorem]{Lemma}

\def\qed{\vbox{\hrule
 \hbox{\vrule\hbox to 5pt{\vbox to 8pt{\vfil}\hfil}\vrule}\hrule}}

\def\endproof{\unskip \nobreak \hskip0pt plus 1fill \qquad \qed \par}
\newcommand{\Span}{\mbox{\rm span}}

\newcommand{\rank}{\mbox{\rm rank$\;$}}

\newcommand{\mb}{\mathbb}
\newcommand{\mc}{\mathcal}

\usepackage{tikz}
\usetikzlibrary{decorations.markings}

\begin{document}

\title{Diagonal Sums of Doubly Stochastic Matrices}

\author{
 Richard A. Brualdi\footnote{Department of Mathematics, University of Wisconsin, Madison, WI 53706, USA. {\tt brualdi@math.wisc.edu}} \\
 Geir Dahl\footnote{Department of Mathematics,  
 University of Oslo, Norway.
 {\tt geird@math.uio.no.} Corresponding author.}
 }
\date{7 December 2020}
\maketitle

\begin{abstract}  Let $\Omega_n$ denote the class of $n \times n$ doubly stochastic matrices (each such matrix is entrywise nonnegative and every row and column sum is 1). We  study the diagonals of matrices in $\Omega_n$. The main question is: which  $A \in \Omega_n$ are such that the diagonals in $A$  that avoid the zeros of $A$ all have the same sum of their entries. We give a characterization of such matrices, and establish several classes of patterns of such matrices. 
\end{abstract}




\section{Introduction}
Let $M_n$ denote the (vector) space of real $n \times n$ matrices and on this space we consider the usual scalar product $A \cdot B=\sum_{i,j} a_{ij}b_{ij}$ for $A, B \in M_n$, $A=[a_{ij}]$, $B=[b_{ij}]$. 

A permutation $\sigma=(k_1, k_2, \ldots, 
k_n)$ of $\{1,2,\ldots,n\}$ can be identified with an $n\times n$ permutation matrix $P=P_{\sigma}=[p_{ij}]$   by defining $p_{ij}=1$, if $j=k_i$,  and $p_{ij}=0$, otherwise. 
If $X=[x_{ij}]$ is  an $n\times n$ matrix, the entries of $X$ in the positions of $X$ in which  $P$ has a 1 is the {\it diagonal} $D_{\sigma}$ of $X$ corresponding to $\sigma$ and $P$, and their sum 
\[
   d_P(X)=\sum_{i=1}^n x_{i,k_i}
   \]
is a {\it diagonal sum} of $X$. Sometimes we refer to the set of positions as a diagonal of $X$. Permutations $\sigma_1,\sigma_2,\ldots,\sigma_k$  of $\{1,2,\ldots,n\}$, and 
 their corresponding permutations matrices,  
are {\it pairwise disjoint} provided no two of them agree in any position; equivalently, no two of  their corresponding permutation matrices have a 1 in the same position. We also then say that the associated diagonals are pairwise disjoint. A {\it zero diagonal} of $X$  is a diagonal of $X$ with  0's in all its positions. Without some restriction on the entries of $X$, diagonal sums can be quite arbitrary.

Let $X=[x_{ij}]$ be a {\it doubly stochastic matrix}. Thus $X$ is a square nonnegative real matrix with all row and column sums equal to 1:
\[x_{ij}\ge 0\ (i,j=1,2,\ldots,n) \mbox{ and }
\sum_{i=1}^nx_{ij}=\sum_{i=1}^nx_{ji}=1\ (j=1,2,\ldots,n).\]
The  set of $n\times n$ doubly stochastic matrices is denoted by $\Omega_n$. $\Omega_n$ is a polytope of dimension $(n-1)^2$ in the space $M_n$ of real matrices of order $n$ with the standard inner product.

A $n \times n$ matrix $A$ is {\em partly decomposable} if suitable permutations of its rows and columns give a matrix 
\[
\left[
 \begin{array}{rr}
   A_1 & A_2 \\
   O & A_3 
 \end{array}
 \right]
\] 
where $A_1$ and $A_3$ are square, non-empty matrices. If $A$ is not partly decomposable, it is called {\em fully indecomposable}. If a doubly stochastic matrix is partly decomposable then, after suitable row and column permutations,  it is a direct sum of fully indecomposable doubly stochastic matrices. Thus in studying properties of doubly stochastic matrices, one usually assumes that they are fully indecomposable.

 Sinkhorn  \cite{Sink} and Balasubramanian \cite{Bala} independently proved the following theorem.
\begin{theorem}
\label{thm:sink} 
Let $X\in\Omega_n$, and let ${\mc D} =\{D_{\sigma_1}, D_{\sigma_2},\ldots, D_{\sigma_k}\}$ be a set of $k$ pairwise disjoint zero diagonals of $X$. Assume that every diagonal of $X$ that is disjoint with the diagonals in $\mc D$ has a constant diagonal sum.  Then all entries of $X$ not on any of the diagonals  in
 $\mc D$ equal $\frac{1}{n-k}$.
\end{theorem}

For a matrix $X$, let $\xi (X)$ be the set of positions in which $X$ has 0's.
Generalizing Theorem \ref{thm:sink},  Achilles \cite{Achilles}   proved the following theorem.


%
\begin{theorem} 
\label{thm:Achilles}
Let $X,Y\in\Omega_n$, and let $Z_X \subseteq \xi(X)$ and $Z_Y \subseteq \xi(Y)$. Assume that all diagonals of $X$ disjoint from $Z_X$ have diagonal sum equal to $\alpha$ and all diagonals of $Y$ disjoint from $Z_Y$ have diagonal sum equal to $\beta$. The following hold:
\begin{itemize}
\item[\rm (i)] If $Z_X \subseteq Z_Y$, then $\alpha\le \beta$.
\item[\rm (ii)] If  $Z_X=Z_Y$, then $\alpha=\beta$ and $X=Y$.
\end{itemize}
\end{theorem}

Theorem \ref{thm:sink} follows from this result by letting $Y$ be the doubly stochastic matrix with the zeros as prescribed by the set  ${\mc D} =\{D_{\sigma_1}, D_{\sigma_2},\ldots, D_{\sigma_k}\}$ of zero diagonals and  with all other elements  equal to  $\frac{1}{n-k}$; see  \cite{Achilles} for details.

\begin{corollary}\label{cor:one} An $n\times n$ doubly stochastic matrix $X$ with a specified set $Z$ of zeros all  of whose diagonal sums  avoiding $Z$  are equal  is uniquely determined.
\end{corollary}

Since their publication more than 40 years ago, the three paper \cite{Achilles,Bala,Sink}  have  not received   attention in the literature. 
In fact, Theorem \ref{thm:Achilles}  is true under more general circumstances as shown next.

\begin{lemma}\label{lem:gen}
 Let $\Gamma$ be a polytope $($in an inner product space$)$ with set of extreme points $W=\{w_1,w_2,\ldots,w_p\}$. Let $u,v\in \Gamma$ and let $X,Y\subseteq W$ such that
\[
   u=\sum_{x\in X}c_x x,\quad \;c_x\ge 0\;  (x\in X), \; \sum_{x\in X} c_x=1
\] 
and
\[
   v=\sum_{y\in Y}c_yy, \quad \;\;\;c_y\ge 0\ (y\in Y), \;\sum_{y\in Y} c_y=1.
\]
Assume that $u\cdot x=a$ for all $x\in X$ and $v\cdot y=b$ for all $y\in Y$. If $Y\subseteq X$, then $a\le b$. If $Y=X$, then  $a=b$ and $u=v$.
\end{lemma}
\begin{proof}
The proof follows the proof in \cite{Achilles}:
\[
u\cdot u=u\cdot \left(\sum_{x\in X}c_x x\right)= \sum_{x\in X} c_x(u\cdot x)=a\sum_{x\in X}c_x=a(1)=a.
\]
Similarly, $v\cdot v=b$. Now suppose that $Y\subseteq X$. Then a similar computation shows that 
$ u\cdot v=a$ and thus
\[0\le (u-v)\cdot (u-v)=u\cdot u -2u\cdot v+v\cdot v=a-2a+b=b-a\]
and hence $a\le b$. If $Y=X$, then we also have that $b\le a$, and hence $a=b$ and $u=v$.
\end{proof}
 
  Let $A$ be an $n\times n$ $(0,1)$-matrix which,  without loss of generality, is assumed to be fully indecomposable. The matrix $A$ defines a face ${\mathcal F}(A)$ of $\Omega_n$ consisting of all  doubly stochastic matrices $X$ with  $\xi(A)\subseteq \xi(X)$ (see \cite{Bru}).  In general, ${\mathcal F}(A)$ contains matrices $X$ where $\xi(A)$ is a proper subset of $ \xi (X)$ such that all the diagonals of $X$  not containing any positions in $\xi (X) $  have equal diagonal sums. The following example \cite{Achilles} is instructive.

\begin{example} {\rm \label{ex:one} 
Let 
\[A=\left[\begin{array}{cccc}
1&1&1&1 \\
1&1&0&0 \\
1&0&1&0 \\
1&0&0&1\end{array}\right].\]
Consider a doubly stochastic matrix $X\in {\mathcal F}(A)$. Then $X$ is of the form
\[X=\left[\begin{array}{cccc}
a+b+c-2&1-c&1-b&1-a \\
1-c&c&0&0 \\
1-b&0&b&0 \\
1-a&0&0&a\end{array}\right]\quad
(0\le a,b,c\le 1, a+b+c-2\ge 0).\]
In order for $X$ to have the same set of 0's as $A$, we must have $0<a,b,c<1$ and $a+b+c-2>0$. There are four diagonals of $X$ that avoid the displayed 0's. A simple computation shows that if the corresponding four diagonal sums are equal,  then $a=b=c$ and  $a+b+c-2=0$ and hence
\begin{equation}\label{eq:X}
X=\frac{1}{3}\left[\begin{array}{cccc}
0&1&1&1\\
1&2&0&0\\
1&0&2&0\\
1&0&0&2\end{array}\right],\end{equation}
We conclude that there does not exist a doubly stochastic matrix with 0's exactly where $A$ has 0's and where all of the diagonals avoiding the positions of the 0's of $A$  have  the same sum.  Thus not every zero pattern of a fully indecomposable $(0,1)$-matrix is realizable as the zero pattern of a doubly stochastic matrix whose diagonal sums avoiding the 0's are constant. Let the  $(0,1)$-matrix $A'$ be obtained from $A$ by replacing the 1 in position $(1,1)$ with a 0. Then $A'\le A$ and $A'$  is realizable as the zero pattern of a doubly stochastic matrix (namely $X$) whose diagonal sums avoiding the 0's have constant value 2.}
\hfill{$\Box$}
\end{example}

For later reference, we note that the matrix $X$ in (\ref{eq:X})  and its zero pattern can be permuted to obtain
\[
\left[\begin{array}{c|c||c|c}
0&1&1&1\\ \hline
1&0&1&0\\ \hline
1&0&0&1\\ \hline\hline
1&1&0&0\end{array}\right]\rightarrow\frac{1}{3} \left[\begin{array}{c|c||c|c}
0&1&1&1\\ \hline
1&0&2&0\\ \hline
1&0&0&2\\ \hline\hline
1&2&0&0\end{array}\right].\]

 A motivation for considering our investigations is the following:
  Given an $n\times n$ matrix $X=[x_{ij}]$, the optimal assignment problem
(OAP) 
asks  for a permutation $(i_1,i_2,\ldots,i_n)$ of $\{1,2,\ldots,n\}$ such that the corresponding diagonal sum in $X$ is maximum. Thus $x_{ij}$ is regarded as representing the value a person (corresponding to row $i$) brings to a job (corresponding to  column $j$).  An assignment of people $1,2,\ldots,n$ to jobs $1,2,\ldots,n$ is denoted by a permutation of $\{1,2,\ldots,n\}$, equivalently, an $n\times n$ permutation matrix. Let $A=[a_{ij}]$ be an $n\times n$ fully indecomposable $(0,1)$-matrix corresponding  to people and jobs as above where an entry $a_{ij}= 0$ is interpreted as person $i$ is not qualified for job $j$, and   an entry $a_{ij}= 1$ is interpreted as person $i$ is qualified for job $j$. Thus the only allowable assignments are those avoiding the 0's in $A$. Assume that $\xi(X)=\xi(A)$. The largest allowable diagonal sum of $X$ solves the OAP, under the restrictions imposed by $A$. Since $A$ is fully indecomposable, so is $X$ and as is well known, there exist diagonal matrices $D_1$ and $D_2$ with entries on the diagonal positive,  such that  $D_1XD_2$ is doubly stochastic. We now assume that $X$ is doubly stochastic, that is, we replace $X$ with $D_1XD_2$.  Thus the values $x_{ij}$ have been normalized so that the total value each person brings to the jobs and the total value of each job equals 1.   If all diagonal sums of $X$ avoiding the 0's of $A$ are equal, then any permissible assignment solves the OAP.

Let $A$ be an $n\times n$ fully indecomposable $(0,1)$-matrix  such that there exists an $n\times n$ doubly stochastic matrix $X$ with $\xi(X)=\xi(A)$, where all diagonals of $X$ disjoint from $\xi(X)$ have equal sum.  Call the matrix $X$ a {\it restricted constant diagonal sum} (abbreviated to {\it RCDS}) doubly stochastic matrix determined by $A$, and call $A$ the {\it pattern of an RCDS doubly stochastic matrix}. Note that if $A$ is the pattern of a RCDS doubly stochastic matrix, so is $PAQ$ for permutation matrices $P$ and $Q$. An analogous assertion holds for $X$. Our goal is to investigate and give methods of construction of RCDS doubly stochastic matrices and their patterns, and  some generalizations as discussed above. Note that if $A=J_n$ so that $\xi(A)=\emptyset$, then $\frac{1}{n}J_n$ is an RCDS doubly stochastic matrix determined by $A$.

\begin{example}
\label{ex:k-ones}
{\rm 
  Let $A$ be a $(0,1)$-matrix with $k$ 1's in each row and column. Define the $n \times n$ matrix $X$ so that $\xi(X)=\xi(A)$ and every nonzero entry in $X$ is $1/k$, i.e., $X=(1/k)A$. Then $X \in \Omega_n$ and every diagonal disjoint from $\xi(X)$ consists of only entries being $1/k$, so that the diagonal sum is $n/k$. Therefore $X$ is the  RCDS doubly stochastic matrix determined by $A$. It is well known that when $A$ has this form, $A$ contains $k$ pairwise disjoint diagonals, so this example is of the type considered in  Theorem \ref{thm:sink}. Note that all nonzero entries in $X$ are equal. Clearly, every matrix with this property must have the form of this example.
  Below is a specific example with $n=4$ and $k=2$, where we indicate a diagonal in boldface:
\[
A=
\left[
\begin{array}{cccc}
{\bf 1}&1&0&0 \\
1&0&{\bf 1}&0 \\
0&{\bf 1}&0&1 \\
0&0&1&{\bf 1}
\end{array}
\right], \;\;\;
X=(1/2)A=
\left[
\begin{array}{cccc}
1/2&1/2&0&0 \\
1/2&0&1/2&0 \\
0&1/2&0&1/2 \\
0&0&1/2&1/2
\end{array}
\right].
\] \endproof
}
\end{example}

We also note that one can determine in polynomial time if a given matrix $X$ is an RCDS doubly stochastic matrix. First, one checks if $X$ is doubly stochastic (trivial), and then one solves two optimal assignment problems, namely 
\[
   \max_{P \le X} \,P \cdot X \;\;\; \mbox{\rm and} \; \;\; \min_{P \le X} \,P \cdot X
\]
where $P$ ranges through permutation matrices $P$ satisfying $P \le X$. We then check if these two optimal values coincide.

The remaining part of the paper is organized as follows. In the next section we give a characterization of RCDS doubly stochastic matrices and a method for their construction. A discussion of a strengthening of the RCDS property is given next. In the two sections that follow we develop certain classes of RCDS doubly stochastic matrices. In the final section we briefly consider the difference of diagonal sums of a doubly stochastic matrix.

Notation: $A$ is a {\em nonnegative matrix} (resp. {\em positive matrix}), and we write $A \ge O$ (resp. $A>0$),  if each entry in $A$ is nonnegative (resp. positive).

\section{Characterization of RCDS matrices}
\label{sec:char}

In this section, we use the duality theorem of linear programming to give a  characterization of   RCDS doubly stochastic matrices which affords a means to construct them.

\begin{theorem}
\label{thm:RCDS-char} 
Let $A=[a_{ij}]$ be a fully indecomposable $(0,1)$-matrix of size $n \times n$ and let $R=(r_1, r_2, \ldots, r_n)$ and $S=(s_1, s_2, \ldots, s_n)$ be the row and column sum vectors of $A$. 

$(i)$ Let $u=(u_1, u_2, \ldots, u_n)$ and $v=(v_1, v_2, \ldots, v_n)$ be real vectors.  Define $Y =Y(u,v)=[y_{ij}]\in M_n$ by $y_{ij}=u_i+v_j$ whenever $a_{ij}=1$, and $y_{ij}=0$ otherwise. Assume that $y_{ij}>0$ whenever $a_{ij}=1$ and that all row and column sums of $Y$ are equal to some positive number $\alpha$, i.e., 
\begin{equation}
\label{eq:eq-line-sums}
\begin{array}{rl}
    u_i r_i+\sum_{j:a_{ij}=1} v_j = \alpha   &(i \le n), \\
    v_j s_j+\sum_{i:a_{ij}=1} u_i = \alpha   &(j \le n).
\end{array}
\end{equation}
Then $X=(1/\alpha)Y(u,v)$ is an RCDS doubly stochastic matrix of $A$.

$(ii)$ Conversely, assume $X$ is an RCDS doubly stochastic matrix of $A$. Then $X=(1/\alpha)Y(u,v)$, as in $(i)$, for some  vectors $u$ and $v$, and $\alpha$ as the common line sum of $Y(u,v)$.
\end{theorem}
\begin{proof}
Since all row and column sums of $Y=Y(u,v)$ are equal to $\alpha$, and $y_{ij}\ge 0$ ($i,j\le n$), $X:=(1/\alpha)Y$ is a doubly stochastic matrix. Clearly $\xi(X)=\xi(Y)=\xi(A)$. Consider a nonzero diagonal in $Y$ corresponding to the permutation $\sigma=(k_1, k_2, \ldots, k_n)$. The associated diagonal sum in $Y$ is 
\[
   d^Y_{\sigma}=\sum_{i=1}^n y_{i,k_i}=\sum_{i=1}^n (u_i+v_{k_i})=
   \sum_{i=1}^n u_i+ \sum_{i=1}^n v_{k_i}=\sum_{i=1}^n u_i+ \sum_{i=1}^n v_i
\]
which is independent of $\sigma$. Thus, all diagonal sums in $Y$, and therefore in $X$, are equal, and $X$ is an RCDS doubly stochastic matrix of $A$. This proves (i). 

To prove (ii), for the given RCDS matrix $X=[x_{ij}]$ consider the linear optimization problem
\begin{equation}
\label{eq: assignment}
\begin{array}{lrl}  \vspace{0.05cm}
 \mbox{\rm minimize}    &\sum_{i,j} x_{ij} y_{ij} \\ \vspace{0.05cm}
 \mbox{\rm subject to}   &\sum_j y_{ij}=1 &(i \le n) \\  \vspace{0.05cm}
                                      &\sum_i y_{ij}=1 &(j \le n) \\  \vspace{0.05cm}
                                      &y_{ij} \ge 0 &(i,j \le n).
\end{array}
\end{equation}
Here we use variables $y_{ij}$ only for those $(i,j)$ such that $x_{ij}\ne 0$; the other $y_{ij}$  can be  assumed to be 0.
It is well-known that the coefficient matrix in (\ref{eq: assignment}) is totally unimodular (see \cite{Schrijver1986}) so there is an optimal solution which is integral; therefore $Y=[y_{ij}]$ is a permutation matrix. Therefore the optimal value $\gamma$ in (\ref{eq: assignment}) is the minimum diagonal sum in the matrix $X$; in fact, all diagonal sums in $X$ equal $\gamma$, by assumption, and  this optimal assignment problem is solved.
By the duality theorem of linear optimization, $\gamma$ is also equal to the maximum value in the dual problem 
\begin{equation}
\label{eq: dual-assignment}
\begin{array}{lrl}  \vspace{0.05cm}
 \mbox{\rm maximize}    &\sum_i u_i + \sum_j v_j \\  \vspace{0.05cm}
 \mbox{\rm subject to}   &u_i+v_j \le x_{ij} &(i,j \le n). \\  \vspace{0.05cm}
 \end{array}
\end{equation}
Here the constraints are present only for  those $(i,j)$ such that $x_{ij}\ne 0$.
 So, there exists $u^*_i$, $v^*_i$ ($i \le n$) such that 
\[
   u^*_i+v^*_j \le x_{ij} \;\;(i,j \le n)\;\; \mbox{\rm  and} \;\;  \sum_i u^*_i + \sum_j v^*_j=\gamma.
\]
To examine the duality relation closer note that if $Y=[y_{ij}] \in \Omega_n$ (so it satisfies the constraints in (\ref{eq: assignment})), then 
\begin{equation}
\label{eq:weakduality}
  \sum_{i,j} x_{ij} y_{ij} \ge \sum_{i,j} (u^*_i+v^*_j) y_{ij}
  =\sum_i u^*_i\sum_j y_{ij} + \sum_j v^*_j\sum_i y_{ij} = \sum_i u^*_i+\sum_j v^*_j=\gamma.
\end{equation}
If $Y$ is a permutation matrix, then the left hand side here is also equal to $\gamma$, and therefore the inequality holds with equality. This means that if $y_{ij}=1$, then $x_{ij}=u^*_i+v^*_j$ ($i,j\le n$).  Since $A$ is fully indecomposable, for every $(i,j)$ outside $\xi(A)$, there exists a permutation matrix with a 1 in that position. Therefore $x_{ij}=u^*_i+v^*_j$ ($i,j\le n$) as desired.
\end{proof}

We remark that the assumption that $A$ is fully indecomposable is only used in part (ii) of the theorem. Moreover, Theorem \ref{thm:RCDS-char} is closely related to Theorem 2.6.3 and Theorem 2.6.4 in \cite{Bapat} where one studies products of diagonals of doubly stochastic matrices. Our result may be obtain by taking the logarithm of these products. The proofs are quite similar in using integrality of the polytope of doubly stochastic matrices and LP duality. However, our proof is shorter for the main part (ii) due to our analysis of equality in (\ref{eq:weakduality}).

The construction given in the theorem is illustrated in the next example.

\begin{example}{\rm \label{ex:root2}
Let $A$ be the $(0,1)$-matrix
\[\left[\begin{array}{c|c|c|c|c|c}
&1&1&1&&\\ \hline
1&&&&1&1\\ \hline
1&&1&&&\\ \hline
1&&&1&&\\ \hline
&1&&&1&\\ \hline
&1&&&&1\end{array}\right].
\]
Choose vectors $u$ and $v$ as indicated below and let $y_{ij}=u_i+v_j$ for each $(i,j) \not \in \xi(A)$ 
\[
\begin{array}{c||c|c|c|c|c|c|} 
u\setminus v&0&0&1&1&1&1 \\ \hline \hline
1&&1&2&2&&\\ \hline
1&1&&&&2&2\\ \hline
2&2&&3&&&\\ \hline
2&2&&&3&&\\ \hline
2&&2&&&3&\\ \hline
2&&2&&&&3 \\ \hline
\end{array}
\]
where every line sum is $5$. Therefore, by Theorem \ref{thm:RCDS-char}, the diagonal sums are equal, in fact equal to $14$.  The corresponding RCDS doubly stochastic matrix is 
\[X=\frac{1}{5}\left[\begin{array}{c|c|c|c|c|c}
&1&2&2&&\\ \hline
1&&&&2&2\\ \hline
2&&3&&&\\ \hline
2&&&3&&\\ \hline
&2&&&3&\\ \hline
&2&&&&3 
\end{array}\right].
\]
}\hfill{$\Box$}\end{example}

\begin{example} 
\label{ex:simplex1} 
{\rm 
Consider the following RCDS doubly stochastic matrix and its construction from vectors $u$ and $v$:
\[
\frac{1}{4}
\left[
\begin{array}{c|c|c|c||c|c|c}
1&&&1&2&&\\ \hline
&&&&2&2&\\ \hline
&&&&&2&2\\ \hline
1&&&1&&&2\\ \hline\hline
2&2&&&&&\\ \hline
&2&2&&&&\\ \hline
&&2&2&&&
\end{array}
\right], 
\hspace{1cm}
\begin{array}{c||c|c|c|c||c|c|c|} 
 u\setminus v&-1&-1&-1&-1&0&0&0 \\ \hline \hline
2&1&&&1&2&&\\ \hline
2&&&&&2&2&\\ \hline
2&&&&&&2&2\\ \hline
2&1&&&1&&&2\\ \hline\hline
3&2&2&&&&&\\ \hline
3&&2&2&&&&\\ \hline
3&&&2&2&&& \\ \hline
\end{array}
\] 
Note that here $v$ has some negative components. By  adding 1 to the components of $u$ and subtracting 1 from the components of $v$, we obtain nonnegative $u$ and $v$. \endproof
}
\end{example}

Theorem \ref{thm:RCDS-char} may be used to construct classes of RCDS doubly stochastic matrices in the following way:
\begin{enumerate}
 \item Let $n \ge 1$, and let $u=(u_1, u_2, \ldots, u_n)$ and $v=(v_1, v_2, \ldots, v_n)$ be real vectors.  Define the matrix $Y =Y(u,v)=[y_{ij}]\in M_n$ by $y_{ij}=u_i+v_j$ and assume that $u$ and $v$ are chosen so that $Y$ is nonnegative (see remark below). 
 \item Choose an $n \times n$ $(0,1)$-matrix $S$ such that the Hadamard product (that is, enrtywise product) $Y \circ S$ has constant positive line (row and column) sums, and let $\alpha$ denote this sum. 
 \item Then $V=(1/\alpha)Y \circ S$ is RCDS doubly stochastic matrix. \endproof
\end{enumerate} 

Clearly, the nontrivial part is step 2 which is to select an  appropriate $S$, that is, to select entries from $Y$ such that all line sums for the selected entries are constant. Then the fact that $V=(1/a)Y \circ S$ is an RCDS doubly stochastic matrix follows from Theorem \ref{thm:RCDS-char}. Moreover, {\em any} such RCDS doubly stochastic matrix may be constructed in this way. Those entries of $Y$ which are not selected could have been negative without altering the conclusion 3. above.

We give an example of this procedure. Let $k, t \ge 1$ be integers and consider  an $n \times n$  matrix
\begin{equation}
\label{eq:RCDS-class1}
V=\frac{1}{tp} \left[
\begin{array}{cccc} 
 tI_k & tI_k& \cdots & tI_k \\ \hline 
\multicolumn{4}{c}{A} 
\end{array}
\right]
\end{equation} 
where the block $tI_k$ occurs $p$ times, $n=kp$, and $A$ is  an $(n-k) \times n$ $(0,1)$-matrix.

\begin{theorem}
\label{thm:RCDS-class1} 
Let $k, t, p \ge 1$ be integers with $t\le k$, and let $n=kp$. Define constant vectors $R=tpJ_{n-k,1}$ and $S=t(p-1)J_{n,1}$, and let $A$ be a $(0,1)$-matrix with row sum vector $R$ and column sum vector $S$.  Then  $V$ in $(\ref{eq:RCDS-class1})$ is an RCDS doubly stochastic matrix.
\end{theorem}
\begin{proof}
Let $u=(t, t, \ldots, t, 1, 1, \ldots, 1)\in \mb{R}^n$ where the first $k$ components are $t$. Let $v$ be the zero vector of length $n$. Then the matrix $Y=Y(u,v)$ (defined above) has $n$ columns, each equal to $u$. So every entry in the first $k$ rows of $Y$ is $t$ and all other entries are 1. 
Thus $V=(1/\alpha)Y \circ S$ where $\alpha=tp$ and $S$ is the $(0,1)$-matrix with ones in the positions of the nonzeros of $V$. By the procedure above  $V$ is an RCDS doubly stochastic matrix. It only remains to show that a matrix $A \in \mc{A}(R,S)$ exists.  Let $R=(r_1, r_2, \ldots, r_n)$ and $S=(s_1, s_2, \ldots, s_n)$. Note that $t(p-1) \le n-k$ as $n-k=kp-k=(p-1)k$ and $t \le k$. Thus $r_i \le n$  and $s_j \le n-k$ for each $i$ and $j$. 
Moreover, $\sum_i r_i=(n-k)tp=ntp-ktp$ and  $\sum_j s_j=n(p-1)t$. 
But then the class $\mc{A}(R,S)$ is nonempty  as both $R$ and $S$ are constant vectors. This may be verified from the Gale-Ryser theorem (see e.g. \cite{RAB91}) as  $S$ is majorized by the conjugate $R^*$ of $R$. 
\end{proof}

\begin{example}
{\rm 
Let $k=3$, $t=2$, $p=2$ and $n=kp=6$. Then the following matrix 
\[
V=(1/4)
\left[
\begin{array}{ccc|ccc} 
2&0&0&2&0&0\\ 
0&2&0&0&2&0\\ 
0&0&2&0&0&2\\ \hline
1  &   1  &   0 &    1  &   1 &    0 \\
0   &  1  &   1  &   1  &   0  &  1 \\
1    & 0  &   1 &    0  &   1  &   1
\end{array}
\right]
\]
is an RCDS doubly stochastic matrix. \endproof
}
\end{example}

Theorem \ref{thm:RCDS-class1} shows that RCDS patterns may be complicated, and as least as complicated as the pattern of $(0,1)$-matrices with constant row sums and constant column sums.

\medskip
We  return to the characterization in Theorem \ref{thm:RCDS-char}. 
Let $A$ be a given $n \times n$ fully indecomposable $(0,1)$-matrix. Let $R(A) = (r_1, r_2, \ldots, r_n)$ and $S(A) = (s_1, s_2, \ldots, s_n)$ be the row sum and column sum vectors of $A$.  Let $D_R$ and $D_S$ be the diagonal matrices with diagonal $R(A)$ and $S(A)$, respectively. Define 
\[
   R_i(A)=\{j: a_{ij} =1\}, \; (i \le n) \;\; \mbox{\rm and}  \;\; C_j(A)=\{i: a_{ij} =1\} \; (j \le n).
\]
Thus $r_i = |R_i(A)|$ and $s_j = |S_j(A)|$ for each $i$ and $j$. In the equations (\ref{eq:eq-line-sums}) in Theorem \ref{thm:RCDS-char} we may assume $\alpha=1$. This gives 
\begin{equation}
\label{eq:RCDS-system1}
 \begin{array}{ll} \vspace{0.1cm}
      r_iu_i + \sum_{j \in R_i(A)} v_j = 1 &(i \le n) \\
      s_jv_j + \sum_{i \in C_j(A)} u_i = 1 &(j \le n) \\
 \end{array}
\end{equation}
which is linear system of equations with $2n$ variables $u_i$, $v_j$ ($i,j \le n$) and $2n$ constraints. 
Rewriting this system in matrix form gives 
\begin{equation}
\label{eq:RCDS-system2}
Hx=e, \;\;\;{\rm where} \;\,
H=
\left[
 \begin{array}{cc} 
      D_R &A   \\
      A^T & D_S
 \end{array}
 \right] \;\mbox{\rm and} \;\;
 x=
 \left[
 \begin{array}{ll} 
      u \\
      v
 \end{array}
 \right].
 \end{equation}
Here  $e$ is the all ones vector (of suitable dimension). 

We observe the surprising fact that the matrix $H$ is equal to the signless Laplacian matrix of the bipartite graph whose biadjacency matrix is $A$.
Therefore, a lot is known on $H$ in terms of spectral properties, e.g., $H$ is  positive semidefinite and singular. The vector $w=(e,-e)$ lies in the null space of $H$. Thus the system (\ref{eq:RCDS-system2}) has $H$ as the coefficient matrix, the right hand side is all ones, and we look for a solution with a certain non-negativity property. We may solve this system using the block structure: 
\[
  D_Ru+Av=e, \; A^Tu+D_Sv=e
\]
which gives $u=D_R^{-1}(e-Av)$, so $A^TD_R^{-1}(e-Av)+D_Sv=e$, i.e., 
\[
 (A^TD_R^{-1}A-D_S)v=A^TD_R^{-1}e-e. 
\]
We next discuss whether this system has a solution or, equivalently, if $Hx=e$ has a solution. 
When $A$ is a $(0,1)$-matrix let $BG(A)$ denote the bipartite graph whose reduced adjacency matrix is $A$. 
 
 The following is a main result on RCDS patterns, based on the discussion above. 
 
 \begin{theorem}\label{th:main}
   Let $A$ be an $n \times n$ fully indecomposable $(0,1)$-matrix. Then the following holds:
   
   $(i)$ There exists $u=(u_1,u_2,\ldots,u_n)$ and $v=(v_1,v_2,\ldots,v_n)$  such that  $(u,v)$ is a solution of the  system $(\ref{eq:RCDS-system2})$. The solution is unique up to adding a constant to each component in $u$ and subtracting the same constant from each component in $v$.
   
   $(ii)$ $A$ is the pattern of an RCDS doubly stochastic if and only if $u_i+v_j>0$ for all $(i,j)$ with $a_{ij}=1$, where $(u,v)$ is an arbitrary solution of $(\ref{eq:RCDS-system2})$. 

 \end{theorem}
 \begin{proof}
  First, as $A$ is fully indecomposable, the bipartite graph  $BG(A)$ is connected. 
  Since $H$ is the signless Laplacian of $BG(A)$, and this graph is connected, it is a known fact that 0 is a simple eigenvalue of $H$. So $H$ has rank $2n-1$.  $Hx=e$ has a  solution provided that $e$ lies in the range (column space) of $H$, so we need to show this. Let $L$ denote the null space of $H$, so $L=\Span \{w\}$ where $w=(e,-e)$. Then $L$ is the orthogonal complement of the row space of $H$. The row space and the column space are equal, as $H$ is symmetric. But 
 \[
    e \cdot w=n-n=0.
 \] 
 Thus $e$ is in $L^{\perp}$, and it follows that $e$ lies in the range of $H$. Hence $Hx=e$ has a  solution, and a general solution is obtained by adding some multiple of the vector $w$. This shows (i). 
 
 Next, assume  $A$ is a pattern of an RCDS. Then, by Theorem \ref{thm:RCDS-char} and the discussion above there exists a solution $x=(u,v)$ of  $Hx=e$. By (i) in this theorem, the solution is unique up to adding a multiple of $w=(e,-e)$  ($w$ spans the null space of $H$), but this does not change the value of $x_{ij}=u_i+v_j$. Thus, we must have that $x_{ij}>0$ (from the initial assumption). The converse implication follows directly from Theorem \ref{thm:RCDS-char}, and the proof is complete.
 \end{proof}
 
 \begin{corollary}\label{cor:sym}
 Let $A$ be an $n\times n$ fully indecomposable, symmetric $(0,1)$-matrix which is the  pattern of an RCDS doubly stochastic matrix. Then there exists a symmetric, RCDS doubly stochastic matrix $X$ with pattern $A$, and a vector $w=(w_1,w_2,\ldots,w_n)$ such that 
 $X=A\circ W$ where $W=[w_{ij}]$ with $w_{ij}=w_i+w_j$  $(1\le i,j\le n)$. 
 \end{corollary}
 
 \begin{proof}
 Let $Z$ be an RCSD doubly stochastic matrix with pattern $A$. Since $A$ is symmetric, $Z^T$ is also 
 an RCSD doubly stochastic matrix with pattern $A$. Hence $X=Z+Z^T$ is  a symmetric,  RCSD doubly stochastic matrix with pattern $A$. By Theorem \ref{th:main} there exists $u=(u_1,u_2,\ldots,u_n)$ and $v=(v_1,v_2,\ldots,v_n)$ such that with  the matrix $Y(u,v)=[y_{ij}]$ where $y_{ij}=u_i+v_j$ (all $i,j$), $X=A\circ Y(u,v)$. But then
 we also have $X=A\circ Y (\frac{u+v}{2},\frac{u+v}{2})$.
 \end{proof}
 
 If $A$ is fully indecomposable, we have a polynomial-time algorithm for deciding if $A$ is an RCDS pattern: One first finds a (near-unique) solution $x=(u,v)$ of $Hx=e$. An efficient way of finding $x$ is by solving 
\[
  (A^TD_R^{-1}A-D_S)v=A^TD_R^{-1}e-e. 
\]
and defining  $u=D_R^{-1}(e-Av)$. Next, we simply check if $x_{ij}=u_i+v_j>0$ for $(i,j)$ with $a_{ij}=1$. If these strict inequalities hold, then $A$ is an RCDS pattern; otherwise, it is not. 

We remark that this algorithm may also be used find some ``random" RCDS patterns. One generates a random $(0,1)$-matrix $A$ and runs the algorithm above. Then, even if $A$ is not an RCDS pattern it may happen that the resulting matrix $X$ is doubly stochstic, but its support is {\em contained} in the support of $A$. This procedure have been used in the example below. 
Finally, we note that the rank of  $H$ satisfies 
$n \le \rank(H) \le 2n-1$ where the lower bound is obtained when $A$ is a permutation matrix.

\medskip
\begin{example}\label{ex:notso}
{\rm Let 
\[
A=
\left[
\begin{array}{cccccc} 
1&0&0&1&1\\ 
0&1&1&1&0\\ 
1&0&0&1&1\\ 
1&1&1&0&0\\ 
1&0&1&1&0 
\end{array}
\right].
\]
Using the procedure above we compute $v$ and $u$: 
 $v=(    0,     0.3,     0.1,     0,     0.25)$ and 
$u=(    0.25,    0.2,    0.25,    0.2,    0.3)$. From this we obtain  
\[
X=
\left[
\begin{array}{rrrrr}
    0.25    &     0   &      0  &  0.25 &   0.5 \\
         0   & 0.5  &  0.3   & 0.2   &      0 \\
    0.25    &     0    &     0  &  0.25  &  0.5 \\
    0.2  &  0.5  &  0.3  &       0   &      0 \\
    0.3   &      0 &   0.4  &  0.3    &     0
\end{array}
\right]
\]    
which is an RCDS doubly stochastic matrix corresponding to the pattern $A$. \endproof
}
\end{example}

\begin{example}
{\rm 
The following are some RCDS patterns  found by the procedure above ($n=5$):
\[
\left[
\begin{array}{r|r|r|r|r}
    1&1&&1&1 \\ \hline
    1&&1&1&1 \\ \hline
    &&1&1&1 \\ \hline
    &1&1&& \\ \hline
    1&1&&&1  
\end{array}
\right], 
\left[
\begin{array}{r|r|r|r|r}
    &&1&1& \\ \hline
    1&&1&1&1 \\ \hline
    &1&&1&1 \\ \hline
    &1&&&1 \\ \hline
    1&&&1&  
\end{array}
\right], 
\left[
\begin{array}{r|r|r|r|r}
    1&1&&&1 \\ \hline
    1&1&&& \\ \hline
    &&1&1&1 \\ \hline
    1&1&1&& \\ \hline
    &&&1&1  
\end{array}
\right], 
\left[
\begin{array}{r|r|r|r|r}
    &&1&1& \\ \hline
    1&1&&&1 \\ \hline
    &&1&&1 \\ \hline
    1&&1&1&1 \\ \hline
    1&1&1&&  
\end{array}
\right]. 
\]    
}
\end{example} \endproof

\section{Compatible classes of Permutation Matrices}
\label{sec:compatible}

In view of our earlier discussion, we now consider, primarily through examples,  the following  general problem (cf. Lemma \ref{lem:gen}). 
Let $A$ be an $n\times n$ fully indecomposable $(0,1)$-matrix and let ${\mathcal P}(A)$ be the set of permutation matrices $P\le A$. Thus ${\mathcal F}(A)=\{X:X\le A, \,X\in \Omega_n\}$  is a face of $\Omega_n$ whose set of extreme points is ${\mathcal P}(A)$.  The cardinality of ${\mathcal P}(A)$, the number of extreme points of ${\mathcal P}(A)$, equals the permanent, per$(A)$, of $A$.

The set ${\mathcal P}(A)$ is  a {\it compatible class} of permutation matrices provided that the scalar product
$Q\cdot \sum_{P\in {\mathcal P}(A)} P$ is a constant $\gamma (A)$ for all $Q\in {\mathcal P}(A)$. We also describe this by saying that $A$ has {\it compatible permutation support} (abbreviated to {\it CPS}). 
If $A$ has CPS, then the  matrix 
\begin{equation}\label{eq:permanent}
\widehat{A}=\frac{1}{\mbox{per}A}\sum \{ P:P\in {\mathcal P}(A)\}\end{equation}
is  a doubly stochastic matrix with constant diagonal sums avoiding the zero set $\xi(A)$ of $A$. Thus the CPS property determines a subclass of RCDS doubly stochastic matrices and provides another possible way to construct such matrices.

In our discussion that follows,  it  is usually more convenient to drop the normalizing factor $ \frac{1}{{\rm per}(A)}$  and to use instead
of (\ref{eq:permanent}) the matrix
\begin{equation}\label{eq:permanent1}
 \widehat{A}=\sum \{ P:P\in {\mathcal P}(A)\}.\end{equation}
In order that ${\mathcal P}(A)$ be a compatible class of permutations, it is necessary and sufficient  that the sum of the permanental minors of $A$ corresponding to the  1's of each permutation matrix $P\le A$ equals a  constant $\gamma(A)$ independent of the permutation matrix $P$.
This is because the permanental minor of a 1 in $A$ counts the number of permutation matrices $Q\le A$ which use that 1. Hence if $A$ has  CPS,  then for each permutation $\sigma=(j_1,j_2,\ldots,j_n)$ of $\{1,2,\ldots,n\}$ with corresponding permutation matrix $P_{\sigma}\le A$,
\begin{equation}
\gamma(A)=\sum_{i=1}^n  \mbox{per} A(i|j_i),  \end{equation}
where $\mbox{per} A(i|j_i)$ is the permanent of the matrix obtained from $A$ by deleting row $i$ and column $j_i$. 
Thus, the matrix  in (\ref{eq:permanent}) is an RCDS doubly stochastic matrix if and only if the sum of the permanental minors of the 1's corresponding to each  permutation matrix $P\le A$ is constant. We remark here that Bapat \cite{Bapat0} proved that if an $n\times n$  fully indecomposable (0,1)-matrix $A$ satisfies that  the permanental minors of all entries (both 0's and  1's) are constant,
then $A=J_n$ or (after row and column permutations) $A=I_n+P_n$, where $P_n$ is the permutation matrix corresponding to the permutation $(2,3,\ldots,n,1)$.

\begin{example} {\rm \label{ex:simple} The following two examples of RCDS doubly stochastic  matrices come from simplex faces of the polytope $\Omega_7$  using the construction given in (\ref{eq:permanent}):
\[\frac{1}{4}\left[\begin{array}{c|c|c|c||c|c|c}
1&&&1&2&&\\ \hline
&&&&2&2&\\ \hline
&&&&&2&2\\ \hline
1&&&1&&&2\\ \hline\hline
2&2&&&&&\\ \hline
&2&2&&&&\\ \hline
&&2&2&&&\end{array}\right] \mbox{ and }
\frac{1}{9}\left[\begin{array}{c|c|c|c||c|c|c}
&&&&3&3&3\\ \hline
&1&1&1&6&&\\ \hline
&1&1&1&&6&\\ \hline
&1&1&1&&&6\\ \hline\hline
3&6&&&&&\\ \hline
3&&6&&&&\\ \hline
3&&&6&&&\end{array}\right].\]
The first has the constant diagonal sum $\frac{13}{4}$; the second has the constant diagonal sum $\frac{7}{3}$.
Both matrices are symmetric. For instance, for the first matrix $u=(1,1,1,1,2,2,2)$ and $v=(0,0,0,0,2,2,2)$ work as in Theorem \ref{th:main}.
}\hfill{$\Box$}\end{example}

\begin{example}\label{ex:notso2}{\rm
We consider again the matrix $A$ in Example \ref{ex:notso}, repeated below.
The permanental minors of its 1's are given in the matrix $M$ where $*$ corresponds to a permanental minor of a 0 and so is not included in our calculations:
\[
A=
\left[
\begin{array}{cccccc} 
1&0&0&1&1\\ 
0&1&1&1&0\\ 
1&0&0&1&1\\ 
1&1&1&0&0\\ 
1&0&1&1&0 
\end{array}
\right], \quad
M=\left[\begin{array}{c|c|c|c|c}
3&*&*&3&6\\ \hline
*&6&4&2&*\\ \hline
3&*&*&3&6\\ \hline
2&6&4&*&*\\ \hline
4&8&4&4&*\end{array}\right].
\]
$M$ has diagonal sums of  $3+4+6+6+4=23$ and $6+6+3+2+4=21$. Hence $M$ does not have equal diagonal sums avoiding the 0's of $A$. While $A$ is the nonzero pattern of an RCDS doubly stochastic matrix, $A$ does not
have CPS. We conclude that CPS is a stronger property than RCDS.
}
\hfill{$\Box$}
\end{example}

There is a cubic bipartite graph $G$ of order 54  contained in the complete bipartite graph $K_{27,27}$, called the {\it Gray graph}, whose automorphism group of cardinality 1296 acts transitively on each set of the bipartition (but not on the complete vertex set)  and transitively on the edge set (an {\it edge-transitive graph} but not a {\it vertex-transitive} graph). The Gray graph 
is the smallest cubic edge-transitive graph which is not vertex-transitive \cite{Mal}.
\begin{center}
\includegraphics[width=150mm]{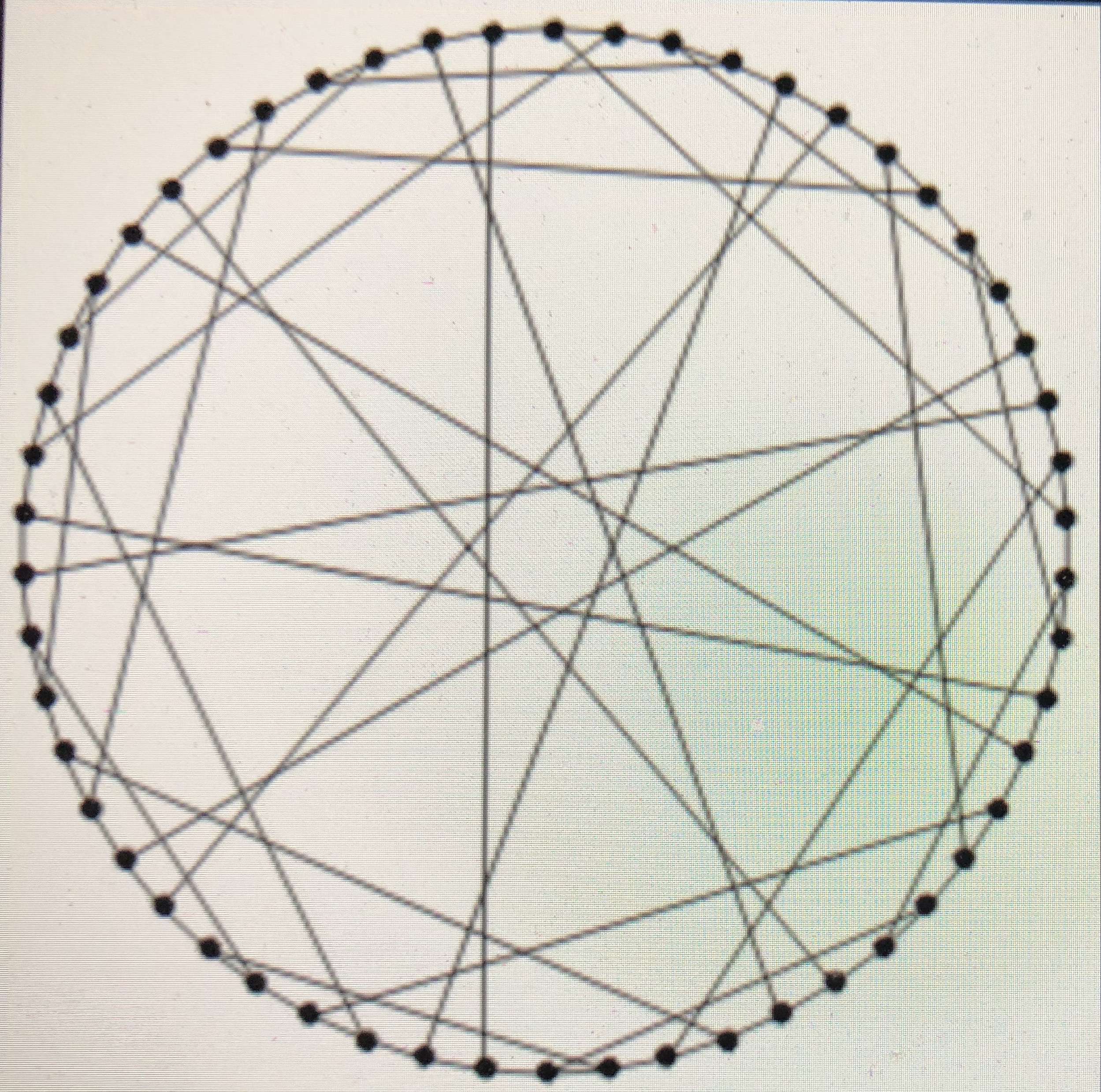}

\medskip
Figure 1: The Gray Graph $G$ (from Wikipedia, en.wikipedia.org).
\end{center}
Let $A$ be the $27\times 27$  biadjacency matrix of $G$ with suitable labeling of its vertices. Using Figure 1, we  have constructed $A$ as shown in Figure 2.

{\footnotesize
\[\left[\begin{array}{c|c|c|c|c|c|c|c|c|c|c|c|c|c|c|c|c|c|c|c|c|c|c|c|c|c|c}
1&1&&&&&1&&&&&&&&&&&&&&&&&&&&\\ \hline
&1&1&&&&&&&&&&&&&&&&&&&&&&1&&\\ \hline
&&1&1&&&&&&&&&&&&&&&&&&&1&&&&\\ \hline
&&&1&1&&&&&&&&&&&&&&1&&&&&&&&\\ \hline
&&&&1&1&&&&&1&&&&&&&&&&&&&&&&\\ \hline
&&&&&1&1&&&&&&1&&&&&&&&&&&&&&\\ \hline
&&&&&&1&1&&&&&&&&&&&&&1&&&&&&\\ \hline
&&&&&&&1&1&&&&&1&&&&&&&&&&&&&\\ \hline
&&&&1&&&&1&1&&&&&&&&&&&&&&&&&\\ \hline
&&&&&&&&&1&1&&&&&&&&&&&&&1&&&\\ \hline
&&&&&&&&&&1&1&&&&&&&&&&&&&&1&\\ \hline
&&&&&&&&&&&1&1&&&&&1&&&&&&&&&\\ \hline
&&1&&&&&&&&&&1&1&&&&&&&&&&&&&\\ \hline
&&&&&&&&&&&&&1&1&&&&&&&1&&&&&\\ \hline
&1&&&&&&&&&&&&&1&1&&&&&&&&&&&\\ \hline
&&&&&&&&&&&1&&&&1&1&&&&&&&&&&\\ \hline
&&&&&&&&&1&&&&&&&1&1&&&&&&&&&\\ \hline
&&&&&&&&&&&&&&&&&1&1&&&&&&&&1\\ \hline
&&&&&1&&&&&&&&&&&&&1&1&&&&&&&\\ \hline
1&&&&&&&&&&&&&&&&&&&1&1&&&&&&\\ \hline
&&&&&&&&&&&&&&&&1&&&&1&1&&&&&\\ \hline
&&&&&&&&1&&&&&&&&&&&&&1&1&&&&\\ \hline
&&&&&&&&&&&&&&&1&&&&&&&1&1&&&\\ \hline
&&&&&&&&&&&&&&&&&&&1&&&&1&1&&\\ \hline
&&&1&&&&&&&&&&&&&&&&&&&&&1&1&\\ \hline
&&&&&&&1&&&&&&&&&&&&&&&&&&1&1\\ \hline
1&&&&&&&&&&&&&&1&&&&&&&&&&&&1
\end{array}\right].\]
}

\medskip
\centerline{Figure 2: Adjacency matrix $A_G$ of the Gray graph $G$.}


 Since the automorphism group of $G$  is edge-transitive, every edge  must be in the same number of perfect matchings since perfect matchings are preserved under automorphisms. Since the perfect matchings containing an edge correspond to permutation matrices $P\le A_G$ containing the 1 corresponding to the edge, each 1 is in the same number of permutation matrices $P\le A_G$. This implies that the permanental minors of the 1's of $A_G$ are all equal. Thus, $A_G$ has  CPS and, in particular, $\widehat{A}_G$ (see (\ref{eq:permanent})) is a RDCS doubly stochastic matrix, and  indeed has a much stronger property. 

There are other classes of $n\times n$ (0,1)-matrices whose 1's have constant permanental minors. The matrices $J_n$ and $I_n+P_n$ as previously discussed (all of whose entries, not just the entries equal to 1) have constant permanental minors.
The $n\times n$ (0,1)-matrix $J_n-I_n$ of all 1's except for 0's on the main diagonal has all the permanental minors of its 1's  equal to the permanent of an $(n-1)\times (n-1)$ $(0,1)$-matrix with exactly $(n-2)$\ 0's where no two of these 0's belong to  the same row or column. The permanental minors of the 0's are also constant but a different constant if $n\ge 4$. Thus $\frac{1}{D_n}(J_n-I_n)$ is a RCDS doubly stochastic matrix where $D_n$ is the permanent of $J_n-I_n$ (the $n$th derangement number).

In general, simplex faces of the polytope $\Omega_n$ correspond to $n\times n$ fully indecomposable (0,1)-matrices which, after permutations of rows and columns, have the form
\begin{equation}\label{eq:simplex*}
A=
\left[\begin{array}{ccc|ccc}
&&&&&\\ 
&A_3&&&A_1&\\ 
&&&&&\\ \hline
&&&&&\\ 
&A_2&&&O&\\ 
&&&&&\end{array}\right]\end{equation}
where, for some $p$ with $0\le p\le n-2$, $A_3$ is an $(n-p)\times (p+1)$ nonzero matrix, and $A_1$ and $A_2^T$ are vertex-edge incidence matrices of trees $T_1$ and $T_2$ and so have exactly two 1's in each column (see Theorem 3.5 of  \cite{Bru}). 
Note that for  Example \ref{ex:simple}, the 1's in $A_3$ are in the rows and columns determined by the pendent vertices of $T_1$ and $T_2$, respectively. Each such row and column of $A_3$ must contain at least one 1 in order that $A$ be fully indecomposable.
The number of permutations $P\le A$ equals the number of 1's in $A_3$ with each such 1 on exactly  one such permutation matrix $P$. Let there be $k$ 1's in $A_3$ so that there are exactly $k$  permutation matrices  $P_1,P_2,\ldots,P_k$ with the $P_i\le A$. Then $X=P_1+P_2+\cdots+P_k$  is a nonnegative integral matrix with pattern $A$ and all its row and column sums equal $k$. Thus $\frac{1}{k}X$ is a doubly stochastic matrix.

Consider a simplex face given by (\ref{eq:simplex*}) in which $A_3$ has at least one 1 in those rows corresponding to the pendent vertices of $T_1$ and in those columns corresponding to the pendent vertices of $T_2$, and in no other positions.  Then $A$ is fully indecomposable.  Let $A_1^*$ be the matrix obtained from $A_1$ by including as a new first column, the column vector which has 1's exactly in those rows in which $A_3$ has a 1. Let $A_2^*$ be defined in a similar way using the columns of $A_3$. Then $A_1^*$ and $A_2^{*T}$ are square, vertex-edge incidence matrices of loopy trees with loops on exactly the pendent vertices of $T_1$ and $T_2$, respectively. Then   $A$ has CPS if and only if the permanental minors of the new 1's in $A_1^*$ are constant and similarly for $A_2^*$. 

\section{Tridiagonal RCDS patterns and trees}
\label{sec:triag-trees}

Let  $A$ be an $n\times n$ $(0,1)$-matrix. Recall that
${\mathcal  F}(A)=\{X \in \Omega_n: X \le A\}$ is the face of the  polytope $\Omega_n$  determined by $A$  and consists of all  those doubly stochastic $n \times n$ matrices that have zeros wherever $A$ has. If $A$ is fully indecomposable, then (see e.g.  \cite{RAB}) the dimension of this face is 
$\sigma(A)-2n+1$
where $\sigma(A)$ is the number of ones in $A$. For a matrix $A$ we let $\rho(A)$ denote its term rank. 

 %
  %
  
%

\begin{lemma}
 \label{lem:RCDS-nec-cond}
 Let $X$ be a an RCDS in $\Omega_n$. Let $X'$ be a $2 \times 2$ submatrix such that $X'>0$ $($entrywise$)$ and its complementary submatrix $X''$ satisfies $\rho(X'')=n-2$. Then the two diagonals in $X'$ have the same sum.
  \end{lemma}
\begin{proof}
 This follows from the fact that a diagonal in $X''$ can be combined with any of the two diagonals in $X'$ to get a diagonal of $X$.
\end{proof}

Lemma \ref{lem:RCDS-nec-cond} gives a necessary condition for a matrix to be an RCDS doubly stochastic matrix. In a certain situation this condition is also sufficient, as we now discuss.

Tridiagonal doubly stochastic matrices were studied in \cite{Dahl04}, and a number of results on this special face $\Omega^t_n$ of $\Omega_n$ were established. Any matrix $X=[x_{ij}] \in \Omega^t_n$ is symmetric and determined by the entries on the superdiagonal, $x_i=x_{i,i+1}$ ($i=1, 2, \ldots, n-1$). We write $X=X(x)$ to indicates this dependency. For instance, for $n=5$ we have 
\[
X(x)=
\left[
\begin{array}{c|c|c|c|c}
1-x_1&x_1&&&\\ \hline
x_1&1-x_1-x_2&x_2& &\\ \hline
&x_2&1-x_2-x_3&x_3& \\ \hline
&&x_3&1-x_3-x_4&x_4 \\ \hline
&&&x_4&1-x_4\\ 
\end{array}
\right]
\]
where the remaining entries are zero.

\begin{theorem}
\label{thm: tridiag}
Let $A=[a_{ij}]$ be the tridiagonal $(0,1)$-matrix of order $n$ with $a_{ij}=1$ whenever $|i-j|\le 1$, and $a_{ij}=0$ otherwise. Then $A$ is the pattern of an RCDS doubly stochastic matrix $X$, and $X=X(x)$ where $x$ is uniquely determined by 
\begin{equation}
\label{eq:tri}
    x_{i-1}+4x_i+x_{i+1}=2  \;\;(i \le n)
\end{equation}
and $x_0=x_{n+1}=0$. The solution $x$ is such that $X(x)$ is doubly stochastic.
\end{theorem}
\begin{proof}
Let  $X=[x_{ij}] =X(x)\in \Omega^t_n$ and assume that $X$ is an RCDS of $A$. Then each entry $x_{ij}$ where $|i-j|\le 1$ is positive, or equivalently, 
\begin{equation}
\label{eq:submatrix-cond}
    x_i>0 \;\; (i \le n-1), \;\;\;   x_{i-1}+x_i<1 \;\; (i \le n), 
\end{equation}
where $x_0=x_n=0$. The only positive $2 \times 2$ submatrices in $X$ are those containing row and column $i$ and $i+1$ ($i=1, 2, \ldots, n-1$). For each of these submatrices the complementary submatrix has full term rank (as the main diagonal contains nonzeros). Therefore, as  $X$ is an RCDS, by Lemma \ref{lem:RCDS-nec-cond}, the following equations must hold
\begin{equation}
\label{eq:subm-1}
   x_i+x_i=(1-x_{i-1}-x_i)+(1-x_i-x_{i+1}),   \;\;(i \le n)
\end{equation}
which gives $(\ref{eq:tri})$. The coefficient matrix $C$ corresponding to the linear equations $(\ref{eq:tri})$ is strictly diagonally dominant, so the  system has a unique solution $x$. Below we show  that $x$ also satisfies (\ref{eq:submatrix-cond}). Moreover, $X=X(x)$ is in fact an RCDS of $A$ which is seen as follows. Consider the identity matrix and let $\alpha$ be its sum in $X$. Any other permutation matrix $P$ with $\xi(P) \subseteq \xi(A)$ is obtained by replacing some $2 \times 2$ submatrices being $I_2$  by $L_2$, see \cite{Dahl04}. However, each such interchange does not change the diagonal sum due to (\ref{eq:subm-1}). Thus all nonzero diagonals in $X$ have the same sum.

It only remains to show  that $x$ satisfies (\ref{eq:submatrix-cond}). This is done by solving $(\ref{eq:tri})$, using Gaussian elimination. Let $C=[c_{ij}]$ be the coefficient matrix, i.e., $c_{ii}=4$ ($i \le n$), $c_{i,i+1}=c_{i+1,i}=1$ ($i=1, 2, \ldots, n-1$), and $c_{ij}=0$ otherwise. Also let $b$ be the $n$-vector with all components being 2. The algorithm is: start with the augmented matrix $\left[ C \; \;b \;\right]$, and for $i=1,2, \ldots, n$, (a) multiply row $i$ by the inverse of the (current) entry $(i,i)$, and then, if $i<n$, (b) subtract that row from the next. Let $F=[f_{ij}]$ be the resulting $n \times (n+1)$ matrix. Then $f_{ii}=1$, and the remaining nonzeros are in positions $(i,i+1)$ ($i<n$)  and in the final column.  Let $h_i$ be the value in position $(i,i)$ of $E$ right after row $i-1$ has been subtracted from row $i$. We then get successively $h_1=4$, $f_{1,n+1} =1/2$, and 
\begin{equation}
\label{eq:recursion}
\begin{array}{rll}
  h_{i} &=4- 1/h_{i-1} &(i=2,3, \ldots, n),\\
  f_{i,i+1}&=1/h_i   &(i=1, 2, \ldots, n-1), \\
  f_{i,n+1}&=(2-f_{i-1,n})/h(i) &(i=2, 3, \ldots, n).
\end{array}  
\end{equation}
Moreover, the solution $x$ is found by back-substitution
\begin{equation}
\label{eq:x-tri}
\begin{array}{ll}
    x_i &=f_{i,n+1}-f_{i,i+1} x_{i+1} \\
        &=(2-f_{i-1,n+1}-x_{i+1})/h(i)  \;\;\;(i=n, n-1, \ldots, 1) 
\end{array}    
\end{equation}
where $x_{n+1}=0$. 

{\em Claim $1:$ $3.7<h_i<3.75$ $(i=2,3,\ldots, n)$.}
Proof of Claim 1: The function $g(h)= 4-1/h$ is strictly increasing for $h>0$ and a computation shows $3.7< g(3.7) < g(3.75) < 3.75$. The claim then follows by induction. 

{\em Claim $2:$  $0.4 \le f_{i,n+1} \le 0.5$ and $0.2 < x_i < 0.5$ $(i=1,2, \ldots, n)$.}
Proof of Claim 2: Assume $0.4 \le f_{i,n+1} < 0.5$ for some $i$. From (\ref{eq:recursion}) and Claim 1 
\[
  f_{i,n+1}=(2-f_{i-1,n})/h(i) \le (2-0.4)/3.7=0.4324<0.5
\]
and 
\[
  f_{i,n+1}=(2-f_{i-1,n})/h(i) \ge (2-0.5)/3.75=0.4.
\]
Thus, by induction, $0.4 \le f_{i,n+1} \le 0.5$ $(i=1,2, \ldots, n)$. Assume $0.2 < x_{i+1} < 0.5$ for some $i$, then  
\[
   x_i=(2-f_{i-1,n+1}-x_{i+1})/h(i) \le (2-0.4-0.2)/3.7=0.3784<0.5
\]
and
\[
   x_i=(2-f_{i-1,n+1}-x_{i+1})/h(i) \ge (2-0.5-0.5)/3.75=0.2667>0.2.
\]
By induction (backward on $i$), $0.2 < x_i < 0.5$ $(i=1,2, \ldots, n)$, and Claim 2 is proved.
Finally, Claim 2 shows that $x$ satisfies (\ref{eq:submatrix-cond}), so  the matrix $X(x)$ is doubly stochastic, as desired.
\end{proof}

One can show further properties of the solution $x$ of the linear system $(\ref{eq:tri})$, but this is not done here. In fact, the exact solution may be found by solving this as a linear second-order difference equation. Note, however, that the modified system where the first and last component on the right hand side is changed to $5/6$, has the solution $x'=(1/3, 1/3, \ldots, 1/3)$. The solution $x$ of $(\ref{eq:tri})$ is close to $x'$ (except near``the two ends").

In an attempt to generalize the result above for tridiagonal matrices we introduce the following notion.  
Let $T$ be a tree on vertices $1,2,\ldots,n$  and let $T^*$ be the {\it  loopy tree} obtained from $T$ by putting a loop at each pendent vertex. 
Let $A=A(T^*)=[a_{ij}]$ be the $n\times n$ adjacency matrix of $T^*$ with 1's on the main diagonal corresponding to the loops.  Note that $A$ is a symmetric (0,1)-matrix.  When $T$ is a path on $n$ vertices we obtain $A$ as in Theorem \ref{thm: tridiag}. With loops allowed, a {\it perfect matching} of $T^*$ is a collection of edges (including loops) that are vertex disjoint and meet all vertices. Perfect matchings of $T^*$ are in one-to-one correspondence with the permutation matrices $P\le A$ and thus their number is the permanent of $A$.

 We have the following lemma.
 
 \begin{lemma}\label{lem:symperm}
 Let $A$ be an $n\times n$ symmetric $(0,1)$-matrix, and let $G$ be the 
  loopy graph whose  adjacency matrix is  $A$.  Then every permutation matrix $P\le A$ is symmetric, that is, corresponds to a perfect matching of $G$, if and only if $G$ does not have any cycles of length $k\ge 3$.
  \end{lemma}
  
  \begin{proof} If there are no permutation matrices $P\le A$, then the lemma is vacuously true. Assume that there is a permutation matrix $P\le A$. If $P$ is the identity matrix $I_n$, then $P$ is symmetric. Now assume that $P\ne I_n$.
  Let the set of vertices of $G$ be $\{1,2,\ldots.n\}$. Then $P$ determines  a bijection $f:\{1,2,\ldots,n\}\rightarrow \{1,2,\ldots,n\}$ such that (a) $\{i,f(i)\}$ is an edge of $G$ for all $i$
 and (b) there exists a $k$ such that $f(k)\ne k$. If $P$ is not symmetric, there exists $i_1\ne i_2$ such that $f(i_1)=i_2$ and $f(i_2)\ne i_1,i_2$.
  Since $f$ is a bijection,  there exists $i_3\ne i_1,i_2$ such that $f(i_2)=i_3$. Continuing like this we see that, since there are only finitely many vertices, there exists distinct  $i_1,i_2,\ldots,i_k$ with $k\ge 3$ such that 
  $f(i_j)=i_{j+1}$ for $1\le j\le k-1$ and $f(i_k)=i_1$. This implies that $G$ contains a cycle of length $k\ge 3$ that gives a permutation cycle of length $k\ge 3$ of $P$ and hence $P$ is not symmetric. The converse is clear since if $P$ has a cycle of length $k\ge 3$, then $P$ is not symmetric. 
  \end{proof}
  
 It  follows from Lemma \ref{lem:symperm} that if $X$ is an $n\times n$  RCDS matrix and $T^*$ is a loopy tree with $\xi(X)=\xi(A(T^*))$, then all diagonals of $X$ avoiding its 0's are symmetric. But not all such $X$ constructed as in (\ref{eq:permanent}) are RCDS matrices.
 
 \begin{example}{\rm \label{ex:symtree}
Let $T$ be the tree with vertices $1,2,\ldots,8$ and edges
\[\{1,2\},\{1,3\},\{1,4\}, \{2,5\},\{2,6\}, \{4,7\},\{4,8\}\}.\]
Let $T^*$ be the loopy tree obtained by putting loops at the pendent vertices $3,5,6,7,8$, and let $A$ be the adjacency matrix of $T^*$. Let $X$ be the sum of the permutation matrices $P\le A$. Then 
\[
X=\left[\begin{array}{c|c|c|c|c|c|c|c}
&2&4&2&&&&\\ \hline
2&&&&3&3&&\\ \hline
4&&4&&&&&\\ \hline
2&&&&&&3&3\\ \hline
&3&&&5&&&\\ \hline
&3&&&&5&&\\ \hline
&&&3&&&5&\\ \hline
&&&3&&&&5\end{array}\right].\]
Then $X$ has two diagonals neither of which contain any 0's and with different sums
\[5+5+2+2+4+3+3+5=29\mbox{ and }
5+3+3+4+4+3+3+5=30.\]

Now let $A$ be the adjacency matrix
\[\left[\begin{array}{c|c|c|c|c|c}
&1&1&1&&\\ \hline
1&&&&1&1\\ \hline
1&&1&&&\\ \hline
1&&&1&&\\ \hline
&1&&&1&\\ \hline
&1&&&&1\end{array}\right]\]
of a loopy tree $T^*$ with 5 perfect matchings.
Then
\[X=\frac{1}{5}\left[\begin{array}{c|c|c|c|c|c}
&1&2&2&&\\ \hline
1&&&&2&2\\ \hline
2&&3&&&\\ \hline
2&&&3&&\\ \hline
&2&&&3&\\ \hline
&2&&&&3\end{array}\right]\]
is an RCDS doubly stochastic with restricted diagonal sums equal to $\frac{14}{5}$.
}\hfill{$\Box$}\end{example}

\begin{example}
\label{ex:star-not-RCDS}
{\rm 
Let $n \ge 2$ and let $x_i\ge 0$ ($2 \le i \le n$). Define the symmetric $n \times n$ matrix $V_n=[v_{ij}]$ by 
$v_{1i}=v_{i1}=x_i$ and $v_{ii}=1-x_i$ ($2 \le i \le n$), and $x_{11}=1-\sum_{i=2}^n x_i$, while all other entries are zero. For instance, when $n=5$ the matrix is 
\[
V_5=
\left[
\begin{array}{c|c|c|c|c}
1-\sum_{i=2}^5 x_i&x_2&x_3&x_4&x_5\\ \hline
x_2&1-x_2& &\\ \hline
x_3&&1-x_3&& \\ \hline
x_4&&&1-x_4& \\ \hline
x_5&&&&1-x_5\\ 
\end{array}
\right].
\]

 $V_n$ is a doubly stochastic matrix when $ x_i \ge 0$ ($2 \le i \le n$) and $\sum_{i=2}^n x_i \le 1$. If all these inequalities are strict, the graph of the matrix is a star with edges $\{1,i\}$ ($2 \le i \le n$) and a loop  in every vertex $i$. Assume that  $X$ is an RCDS matrix. By Lemma \ref{lem:RCDS-nec-cond}, the following equations must hold
\[
   x_i+x_i=(1-x_i)+(1-\sum_{k=2}^n x_k),   \;\;(2 \le i  \le n)
\]
or equivalently
\[
   4x_i+\sum_{k \not = i} x_k = 2,   \;\;(2\le i \le n).
\]
The unique solution is $x_i=2/(n+2)$ ($2\le i \le n)$. However, $\sum_{i=2}^n x_i=2(n-1)/(n+2)$ which is $\le 1$ if and only if $n \le 4$. Thus, when $n \ge 5$, the corresponding matrix is not doubly stochastic  (the entry in position $(1,1)$ is negative) and this proves  that there is no RCDS doubly stochastic matrix with this pattern. The remaining cases $n \le 4$ correspond to the following matrices
\[
V_2=
\left[
\begin{array}{c|c}
1/2&1/2\\ \hline
1/2&1/2\\ 
\end{array}
\right], \;
V_3=
\left[
\begin{array}{c|c|c}
1/5&2/5&2/5\\ \hline
2/5&3/5&\\ \hline
2/5&&3/5 \\ 
\end{array}
\right], \;
V_4=
\left[
\begin{array}{c|c|c|c}
0&1/3&1/3& 1/3\\ \hline
1/3&2/3&&\\ \hline
1/3&&2/3& \\ \hline
1/3&&&2/3 \\ 
\end{array}
\right]
\]
and it is easy to check that each of these is a RCDS doubly stochastic matrix.
\endproof
}
\end{example} 

To summarize, we have shown that  when the (loopy) tree $T$ is a path, the corresponding adjacency matrix $A$ is the pattern of a RCDS doubly stochastic, while when $T$ is a star, $A$ is not an RCDS pattern for $n \ge 5$. Moreover, we gave two other trees, in either of these two categories. Based on this, one might guess that the loopy trees that correspond to RCDS patterns must satisfy some constraint on the maximum degree. It is an open question to characterize the loopy trees that correspond to RCDS patterns.

\section{More classes of  RCDS matrices}
\label{sec:ext}

From Theorem \ref{thm:Achilles} we immediately obtain the following result.

\begin{corollary}
\label{cor:RCDS-J}
 The only RCDS doubly stochastic matrix without any  zeros is $(1/n)J_n$.
\end{corollary}

\begin{proof} We give an alternative proof of this corollary without applying Theorem \ref{thm:Achilles}.
 Let $A=[a_{ij}]$ be a RCDS doubly stochastic matrix with no zeros.   
  Consider a $2 \times 2$ submatrix $B$ of $A$. By Lemma \ref{lem:RCDS-nec-cond}, and as $A$ has no zeros, the two diagonals in $B$ have the same sum. 
  Let $i \le n$ and assume $a_{ij}=a_{ik}$ for some $j \not = k$. Then column $j$ and column $k$ must be equal. Assume $A \not = (1/n)J_n$. Since $A$ is doubly stochastic, at most $n-2$ columns are equal. Consider two unequal columns, say columns $j$ and $j'$. Then 
$a_{ij} \not = a_{i,j'}$ ($i \le n$). Then there are $i,i'$ such that 
 $a_{ij} > a_{ij'}$ and $a_{i'j} < a_{i'j'}$.  So, in the submatrix of $A$ consisting of rows $i$, $i'$ and columns $j$, $j'$, we have 
 \[
     a_{ij}+a_{i'j'}>a_{ij'}+a_{i'j}, 
 \]
 a contradiction.  This shows that the only possibility is $A=(1/n)J_n$.
\end{proof}

We next identify another class of RCDS doubly stochastic matrices.
Let $r$, $s$ and $n$ be positive integers satisfying $s<r<n$, and define the matrix $X=X^{(r,s,n)}=[x_{ij}] \in M_n$ by 
\begin{equation}
 \label{eq:rcds-zero-block}
    x_{ij} = 
    \left\{
      \begin{array}{ll}
          1/r & (i\le r, \,j \le s) \\
          (r-s)/(r(n-s))   &(i\le r, \,s<j \le n) \\
          0     &(r<i\le n, \,j \le s) \\
          1/(n-s)  &(r<i\le n, \,s<j \le n). \\
      \end{array}    
    \right.
\end{equation}

\begin{proposition}
\label{pr:RCDS-zero-block}
 $X^{(r,s,n)}$ is an RCDS doubly stochastic matrix for each $s<r<n$.
\end{proposition}
\begin{proof}
All entries of $X=X^{(r,s,n)}$ are nonnegative. For each $j\le s$ the $j$'th column sum is $r \cdot(1/r)=1$, and for each $s<j\le n$ the $j$'th column sum is
\[
  r \cdot (r-s)/(r(n-s)) + (n-r) \cdot 1/(n-s) =(r-s+n-r)/(n-s)=1.
\]
Next, for each $i\le r$ the $i$'th row sum is 
\[
  s \cdot (1/r) + (n-s) \cdot (r-s)/(r(n-s)) =(s+r-s)/r=1.
\]
while for $r<i\le n$ the $i$'th row sum is $(n-s)\cdot (1/(n-s))=1$. Therefore $X$ is doubly stochastic. 

Consider a diagonal $D$ in $X$ avoiding zeros. From the first $s$ columns it contains $s$ times the entry $1/r$. Next $D$ contains additionally $r-s$ entries in the $r$ first rows and each such entry must be in the last $(n-s)$ columns and is therefore equal to $(r-s)/(r(n-s))$. Finally, $D$ contains $n-r$ additional entries in the last $n-s$ columns but from the last $n-r$ rows. Each such entry is $1/(n-s)$. Thus, the diagonal $D$ contains
\[
   1/r \;(s \;\mbox{\rm times}), \mbox{\rm and}\;
   (r-s)/(r(n-s)) \;((r-s) \;\mbox{\rm times}), \mbox{\rm and}\;
   1/(n-s) \;((n-r) \;\mbox{\rm times}).
\]
So, all diagonals contain exactly the same numbers, and their sums are equal. This shows that $X$ 
is an RCDS doubly stochastic matrix.
\end{proof}

Proposition \ref{pr:RCDS-zero-block} provides a construction of a fully indecomposable $n\times n$ RCDS matrix whose zeros form a $p\times q$ submatrix for any positive $p$ and $q$ with $p+q\le n-1$.

\begin{example}
\label{ex:lower-staircase}
{\rm 
Consider the  matrix $X =X^{(3,2,6)}\in \Omega_6$ given by
\[
X=
\left[
\begin{array}{c|c|c|c|c|c} 
1/3&1/3&1/12&1/12&1/12&1/12\\ \hline
1/3&1/3&1/12&1/12&1/12&1/12\\  \hline
1/3&1/3&1/12&1/12&1/12&1/12\\  \hline
0&0&1/4&1/4&1/4&1/4\\  \hline
0&0&1/4&1/4&1/4&1/4\\  \hline
0&0&1/4&1/4&1/4&1/4 
\end{array}
\right].
\]
Any diagonal avoidning zeros must contain the entries $1/3$, $1/3$ (from the first two columns), $1/12$ (from one of the first three rows), and $1/4$, $1/4$ and $1/4$ (from three of the last columns). \endproof
}
\end{example}

The class discussed in the previous proposition can be extended to matrices with staircase pattern.
A matrix is {\em constant } if all entries are equal (so it is a multiple of the all ones matrix). Let $k \ge 3$. Let $X$ be a $n \times n$ matrix of the form
\begin{equation}
\label{eq:zig-zag-matrix}
X=
\left[
\begin{array}{c|c|c|c|c|c} 
X_1&X_2&&&&\\ \hline
&X_3&X_4&&& \\ \hline 
&&X_5&X_6&& \\ \hline
&&&\ddots &\ddots& \\ \hline
&&&&X_k&X_{k+1}
\end{array}
\right]
\end{equation}
where $X_i$ ($i \le k+1$) are constant matrices and open space indicates a zero matrix. We permit $X_{k+1}$ to be void. The constant associated with $X_i$  is denoted by $c(X_i)$  ($i \le k+1$), so $X_i=c(X_i)J$. We call $X$ a {\em zig-zag matrix}. Let $X_i$ have dimension $r_i \times s_i$ ($i \le k+1$). 
Consider the following conditions on the dimensions
\begin{equation}
\label{eq:dim-cond}
    \sum_{i=1}^t s_i <\sum_{i=1}^t r_i <\sum_{i=1}^{t+1} s_i \;\;\;(t \le k).
\end{equation}

\begin{example}
\label{ex:lower-staircase2}
{\rm 
The following matrix $X$ is a zig-zag doubly stochastic matrix:
\[
X=
\left[
\begin{array}{c|c|c|c} 
X_1&X_2&&\\ \hline
&X_3&X_4& \\ \hline 
&&X_5&X_6 
\end{array}
\right]
=
\left[
\begin{array}{c||c|c||c|c||c} 
1/2&1/4&1/4&0&0&0\\ \hline
1/2&1/4&1/4&0&0&0\\  \hline \hline
0&1/4&1/4&1/4&1/4&0\\  \hline
0&1/4&1/4&1/4&1/4&0\\  \hline \hline
0&0&0&1/4&1/4&1/2\\  \hline
0&0&0&1/4&1/4&1/2 
\end{array}
\right].
\] \endproof
}
\end{example}

\begin{theorem}
\label{thm:RCDS-zero-block}
 Let $X$ be a doubly stochastic zig-zag matrix satisfying $(\ref{eq:zig-zag-matrix})$ and $(\ref{eq:dim-cond})$.    Then $X$ is an RCDS matrix.
\end{theorem}
\begin{proof}
  Consider a diagonal $D$ in $X$ with no zeros.  
 Since $r_1>s_1$, we get $c(X_1)=1/r_1$ and $D$ contains one entry from each of the first $s_1$ columns of $X$.  So, $D$ has $s_1$ entries from $X_1$. Moreover, as $s_1<r_1<s_1+s_2$,  $D$ contains  the remaining $r_1-s_1$ entries in the $r_1$ first rows in the last $(n-s_1)$ columns, i.e., $D$ has $r_1-s_1$ entries in $X_2$, all equal to $c(X_2)$.
The remaining $s_2-(r_1-s_1)=s_1+s_2-r_1>0$ entries that $D$ contains in the columns in $X$ corresponding to  $X_2$ must all lie in $X_3$. This again implies that $D$ has a a fixed number of entries in $X_4$, etc. So, continuing like this, any diagonal contains a fixed number of entries in $X_i$ ($i \le k$). Therefore each such diagonal contains the same numbers, and then clearly their sums are equal.
\end{proof}

We proceed and identify a large class of RCDS matrices. They are constructed inspired by the procedure after Theorem \ref{thm:RCDS-char}. Let $\mc{A}_{k,n}$ be the class of $n \times n$ $(0,1)$-matrices with $k$ ones in every row and column. Consider a $n \times n$ $(0,1)$-matrix
\begin{equation}
 \label{eq:class3}
 A=[a_{ij}]=
\left[
\begin{array}{c|c} 
A_{1}&A_{2}\\ \hline
A_{3}&A_{4}\\ 
\end{array}
\right]
\end{equation}
where each of the blocks $A_i$ has size $p \times p$ ($i \le 4$) so $n=2p$. 
Assume that  $A_i \in \mc{A}_{k_i,p}$ where $k_i \le p$ $(i \le 4)$. 
 
Now, let  $u=(u_i)=(a_1, \ldots, a_1, a_2, \ldots, a_2)$ and $v=(v_j)=(b_1, \ldots, b_1, b_2, \ldots, b_2)$ where the first $p$ components are equal, in each of these vectors.  Let $X'=[x'_{ij}]$ be the $n \times n$ matrix where $x'_{ij}=u_i+v_j$  when $a_{ij}=1$, and $x'_{ij}=0$ otherwise ($i,j\le n$).   
The $i$'th row sum in  $X'$ is
\[
r_i(X')=
\left\{
\begin{array}{ll}
    k_1(a_1+b_1)+k_2(a_1+b_2)  & (i \le p)\\
    k_3(a_2+b_1)+k_4(a_2+b_2)  & (i>p),
 \end{array}
 \right.   
\]
and the $j$'th column sum in $X'$ is 
\[
s_j(X')=
\left\{
\begin{array}{ll}
    k_1(a_1+b_1)+k_3(a_2+b_1),  & (j \le p)\\
    k_2(a_1+b_2)+k_4(a_2+b_2)  & (j>p).
 \end{array}
 \right.   
\]
We want all these four sums to be equal; this gives three equations where one is redundant. By some simplification we get the equivalent system 
\begin{equation}
 \label{eq:2-4-system}
\begin{array}{ll}
    k_1(a_1+b_1)=k_4(a_2+b_2), \\
    k_2(a_1+b_2)=k_3(a_2+b_1).
\end{array}    
\end{equation}
The next theorem is obtained by finding a solution to these two  equations.

\begin{theorem}
\label{thm:class3}
 Let $A=[a_{ij}]$ be as in $(\ref{eq:class3})$ where $A_i \in \mc{A}_{k_i,p}$ and $k_i \le p$ $(i \le 4)$ satisfy $k_1+k_4=k_2+k_3$. Let $X'=[x'_{ij}]$ be the $n \times n$ matrix where $x'_{ij}>0$ precisely when $a_{ij}=1$ and 
 \[
x'_{ij}=
\left\{
\begin{array}{ll}
   k_4 ,  & (i,j \le p)\\
   k_3 ,  & (i \le p, \,j>p)\\
   k_2 ,  & (i > p, \, j\le p)\\
   k_1 ,  & (i > p,\, j>p).\\
  \end{array}
 \right.   
\]
Define $X=(1/(k_1k_4+k_2k_3)) \cdot X'$. Then $X$ is an RCDS doubly stochastic matrix and $A$ is the corresponding RCDS pattern.
\end{theorem}
\begin{proof}
The equations (\ref{eq:2-4-system}) is a linear system of two equations in four unknows $a_1$, $a_2$, $b_1$, $b_2$. 
One way to find a solution is to solve (if possible) 
\[
  a_1+b_1=k_4, \; a_2+b_2=k_1, \; a_1+b_2=k_3, \; a_2+b_1=k_2 \; 
\]
or 
\[
  a_1=k_4-b_1, \; a_2=k_1-b_2, \; (k_4-b_1)+b_2=k_3, \; (k_1-b_2)+b_1=k_2.
\]
The last two equations  give
\[
   b_2-b_1=k_3-k_4, \; \;b_2-b_1=k_1-k_2
\]
which is consistent as  $k_1+k_4=k_2+k_3$. Choose
\[
   b_1=0, \;\; b_2=k_3-k_4, \;\; a_1=k_4, \;\; a_2=k_2.
\]
This is a solution of (\ref{eq:2-4-system}). The corresponding matrix $X'$ is then as described in the theorem, and, by the discussion before the theorem, all row and column sums in $X'$ are equal to some number $\alpha$. We compute $\alpha=k_1k_4+k_2k_3$. Therefore $X=(1/\alpha)X'$ is doubly stochastic and has the same pattern as $X$. Finally, by Theorem \ref{thm:RCDS-char}, $X$ has all diagonals sums that avoid zeros in $A$ equal, so the conclusion of the theorem follows.
\end{proof}

Note that it should not be difficult to find {\em all} solutions of the (\ref{eq:2-4-system}), and possible more RCDS patterns.

\begin{example}
\label{ex:class3}
{\rm 
Let $k_1=1$, $k_2=2$, $k_3=3$, $k_4=4$, $p=5$ and $n=2p=10$. Then, by Theorem \ref{thm:class3}, the following matrix is an RCDS doubly stochastic matrix
\[
X=(1/10)
\left[
\begin{array}{c|c|c|c|c||c|c|c|c|c} 
 &&&4&&&3&3&& \\ \hline
 &&&&4&&3&&&3 \\ \hline
 &4&&&&3&&&3& \\ \hline
 &&4&&&&&&3&3 \\ \hline
 4&&&&&3&&3&& \\ \hline \hline
 &2&2&2&&1&1&&1&1 \\ \hline
 2&2&&&2&&1&1&1&1 \\ \hline
 2&&2&&2&1&1&1&&1 \\ \hline
 &&2&2&2&1&&1&1&1 \\ \hline
 2&2&&2&&1&1&1&1& \\ 
\end{array}
\right].
\]
}
\end{example} \endproof

\section{The diagonal width}
\label{sec:diag-width}

To conclude, we mention a  generalization of  the concept of having equal diagonal sums. Let $X \in \Omega_n$. Define the {\em diagonal width} DW$(X)$ of $X$ by
\begin{equation}
 \label{eq:DW-def}
   {\rm DW}(X)=\max\{d^X_P: P \in \mc{P}_n, \;\xi(X) \subseteq \xi(P)\}- \min\{d^X_Q: Q \in \mc{P}_n, \; \xi(X) \subseteq \xi(Q)\}. 
\end{equation}
This  is the maximum difference between two diagonal sums in $X$, where the diagonals avoid the zeros of $X$. Thus, $X$ is an RCDS (doubly stochastic matrix) if and only if DW$(X)=0$. 


The following result characterizes DW$(X)$, and may be used to analyze the diagonal width. The proof uses the duality ideas  in the proof of Theorem \ref{thm:RCDS-char}.

\begin{theorem}
\label{thm:DSP-char} 
Let $X \in \Omega_n$. Then 
\begin{equation}
\label{width-exp}
   {\rm DW}(X)=\theta^*(X)-\theta_*(X)
\end{equation}
where 
\[
 \begin{array}{rcr} \vspace{0.2cm}
   \theta^*(X)&= &\min\{\sum_i u_i + \sum_j v_j: u_i+v_j \ge x_{ij} \;((i,j) \not \in \xi(X))\}, \\
    \theta_*(X)&= &\max\{\sum_i u_i + \sum_j v_j: u_i+v_j \le x_{ij} \;((i,j) \not \in \xi(X))\}, 

 \end{array}
\]
\end{theorem}
\begin{proof}
Recall from the proof of Theorem \ref{thm:RCDS-char} that the linear optimization problem
\begin{equation}
\label{eq: assignment1}
\begin{array}{lrl}  \vspace{0.05cm}
 \mbox{\rm minimize}    &\sum_{i,j} x_{ij} y_{ij} \\ \vspace{0.05cm}
 \mbox{\rm subject to}   &\sum_j y_{ij}=1 &(i \le n) \\  \vspace{0.05cm}
                                      &\sum_i y_{ij}=1 &(j \le n) \\  \vspace{0.05cm}
                                      &y_{ij} \ge 0 &((i,j) \not \in \xi(X)).
\end{array}
\end{equation}
has optimal value $\min\{d^X_P: P \in \mc{P}_n\}$, the minimum diagonal sum in $X$. Here there is a  variable $y_{ij}$ for each $(i,j) \not \in \xi(X)$, $i,j \le n$. Moreover, by the duality theorem 
\[
   \min\{d^X_P: P \in \mc{P}_n\} =\theta^*(X)
\]
where $\theta_*(X)$ is the optimal value of the dual problem 
\begin{equation}
\label{eq: dual-assignment1}
\begin{array}{lrl}  \vspace{0.05cm}
 \mbox{\rm maximize}    &\sum_i u_i + \sum_j v_j \\  \vspace{0.05cm}
 \mbox{\rm subject to}   &u_i+v_j \le x_{ij} &((i,j) \not \in \xi(X)). \\  \vspace{0.05cm}
 \end{array}
\end{equation}
We use this duality relation to find an expression for $\max\{d^X_P: P \in \mc{P}_n\}$:
\[
\begin{array}{ll} \vspace{0.05cm}
  \max&\{d^X_P: P \in \mc{P}_n\} 
    = \max\{\langle X,P\rangle: X \in \Omega_n\}           \\  \vspace{0.05cm}
    &= -\min\{\langle -X,P\rangle: X \in \Omega_n\}           \\  \vspace{0.05cm}
    &=-\max \{\sum_i u_i + \sum_j v_j:  u_i+v_j \le -x_{ij} \;((i,j) \not \in \xi(X))\}   \\  \vspace{0.05cm}
    &=-\max \{-\sum_i (-u_i) + \sum_j (-v_j):  (-u_i)+(-v_j) \ge  x_{ij} \;((i,j) \not \in \xi(X))\}   \\  
    &=\min \{\sum_i u_i + \sum_j v_j:  u_i+v_j \ge  x_{ij} \;((i,j) \not \in \xi(X))\}    \\  \vspace{0.05cm}
    &= \theta^*(X)
\end{array}
\]
by change of variables. This gives (\ref{width-exp}). 
\end{proof}

\begin{example} 
\label{ex:DW-1} 
{\rm 
Let $X \in \Omega_2$ so
\[
X=
\left[
\begin{array}{cc}
  1-\alpha & \alpha \\
    \alpha & 1-\alpha
\end{array}
\right].
\] 
Assume $\alpha \ge 1/2$ (by symmetry). If $\alpha=1$, clearly  DW$(X)=0$, so assume $\alpha<1$. Then there are two diagonals and their sums are $2(1- \alpha)$ and $2\alpha$. Then DW$(X)=2\alpha-2(1-\alpha)=4\alpha-2$.
\endproof
}
\end{example}

The previous theorem may be used to find upper bounds on DW$(X)$ for a given $X \in \Omega_n$, as the following corollary shows. 

\begin{corollary}
\label{cor:DW-ub}
Let $X \in \Omega_n$. 
Let  $u_i$, $v_i$ and $u'_i$, $v'_i$  be such that 
\[
    u_i+v_j \le x_{ij} \le u'_i+v'_j \;\;\;((i,j) \not \in \xi(X)). 
\]
Then 
\[
   {\rm DW}(X)\le \sum_i (u_i-u'_i) + \sum_j (v_i-v'_i). 
\]
\end{corollary}
\begin{proof}
  This follows directly from Theorem \ref{thm:DSP-char}.
\end{proof}

\end{document}